\documentclass[a4paper,12pt]{amsart}
\usepackage[margin=2.5cm]{geometry}
\usepackage[utf8]{inputenc}
\usepackage[T1]{fontenc}
\usepackage[charter]{mathdesign}
\usepackage{hyperref}
\hypersetup{
bookmarks=true,
pdfpagemode=UseNone,
pdfstartview=FitH,
pdfdisplaydoctitle=true,
pdfborder={0 0 0}, 
pdftitle={Uniformity in the Wiener-Wintner theorem},
pdfauthor={Tanja Eisner and Pavel Zorin-Kranich},
pdfsubject={Mathematics Subject Classification (2010): 37A45, 37A30, 28D05},
pdfkeywords={Wiener-Wintner theorem, nilsequence},
pdflang=en-US
}
\usepackage[msc-links,initials,backrefs,alphabetic]{amsrefs} 

\numberwithin{equation}{section}
\newtheorem{theorem}[equation]{Theorem}
\newtheorem{lemma}[equation]{Lemma}

\newtheorem{corollary}[equation]{Corollary}
\theoremstyle{definition}
\newtheorem{definition}[equation]{Definition}
\theoremstyle{remark}
\newtheorem{remark}[equation]{Remark}
\newtheorem{example}[equation]{Example}
\newtheorem*{ack}{Acknowledgment}

\newcommand*{\N}{\mathbb{N}}
\newcommand*{\Z}{\mathbb{Z}}
\newcommand*{\C}{\mathbb{C}}
\newcommand*{\R}{\mathbb{R}}
\newcommand*{\T}{\mathbb{T}}
\newcommand*{\dif}{\mathrm{d}}
\newcommand*{\im}{\mathrm{im}}
\def\<{\left\langle}
\def\>{\right\rangle}
\newcommand*{\inv}{^{-1}}

\newcommand*{\E}{\mathbb{E}}
\newcommand*{\from}{\colon}
\newcommand{\aveN}{\frac{1}{N}\sum_{n=1}^N}
\newcommand{\aveFn}{\frac{1}{|\Phi_N|}\sum_{n\in \Phi_N}}

\newcommand{\m}{\mathbf{m}}
\newcommand*{\ol}[1]{\overline{#1}}
\newcommand{\veps}{\varepsilon}
\newcommand*{\Gb}[1][]{G_{\bullet #1}}
\newcommand*{\poly}[1][\Gb]{P(\Z,#1)}
\newcommand*{\HKZ}{\mathcal{Z}}

\begin{document}
\title{Uniformity in the Wiener-Wintner theorem for nilsequences}
\author{Tanja Eisner}
\address
{Korteweg-de Vries Institute for Mathematics\\
University of Amsterdam\\
P.O.\ Box 94248\\
1090 GE Amsterdam\\
The Netherlands}
\email[T.~Eisner]{t.eisner@uva.nl}
\author{Pavel Zorin-Kranich}
\email[P.~Zorin-Kranich]{zorin-kranich@uva.nl}
\urladdr{http://staff.science.uva.nl/~pavelz/}
\begin{abstract}
We prove a uniform extension of the Wiener-Wintner theorem for nilsequences due to Host and Kra and a nilsequence extension of the topological Wiener-Wintner theorem due to Assani.
Our argument is based on (vertical) Fourier analysis and a Sobolev embedding theorem.
\end{abstract}
\date{\today}
\maketitle

\section{Introduction}\label{sec:intro}
Let $(X,\mu)$ be a probability space and let $T:X\to X$ be an invertible measure preserving transformation.
The classical Wiener-Wintner theorem \cite{MR0004098} asserts that for every $f\in L^1(X,\mu)$ there exists a subset $X'\subset X$ with full measure such that the weighted averages
\begin{equation}\label{ave-ww}
  \aveN f(T^nx) \lambda^n
\end{equation}
converge as $N\to\infty$ for every $x\in X'$ and every $\lambda$ in the unit circle $\T$.

Over the years this theorem has been improved and generalized in many directions.
For example, Lesigne \cites{MR1074316,MR1257033} proved that the weights $(\lambda^n)$ can be replaced by polynomial sequences of the form $(\lambda_1^{p_1(n)}\cdots \lambda_k^{p_k(n)})$, $\lambda_j\in \T$, $p_j\in \Z[X]$ (or, equivalently, $(e^{2\pi i p(n)})$, $p\in\R[X]$).
More recently Host and Kra \cite{MR2544760}*{Theorem 2.22} showed that this can be enlarged to the class of nilsequences.

In a different direction, Bourgain's uniform Wiener-Wintner theorem \cite{MR1037434} asserts convergence of the averages \eqref{ave-ww} to zero for $f$ orthogonal to the Kronecker factor \emph{uniformly in $\lambda$}, cf.~Assani \cite{MR1995517}.
A joint extension of this result and Lesigne's polynomial Wiener-Wintner theorem has been obtained by Frantzikinakis \cite{MR2246591}.
In the same spirit, our main result is a uniform version of the Wiener-Wintner theorem for nilsequences.

Let $G$ be a nilpotent Lie group with a cocompact lattice $\Gamma$.
The compact manifold $G/\Gamma$ together with the Haar measure on it is called a \emph{nilmanifold}.
Using the universal covering we may and will assume that the connected component of the identity $G^{o}$ is simply connected.
Let further $\Gb$ be a $\Gamma$-rational filtration of length $l$ on $G$ and $\poly$ be the group of $\Gb$-polynomials (we recall these notions in Section~\ref{sec:preliminary}).
Then for every polynomial $g\in\poly$ and $F\in C(G/\Gamma)$ we call the sequence $(F(g(n)\Gamma))_{n}$ a \emph{basic $l$-step nilsequence}.
An \emph{$l$-step nilsequence} is a uniform limit of basic $l$-step nilsequences (which are allowed to come from different nilmanifolds and filtrations).

Nilsystems (i.e.\ rotations on nilmanifolds) and nilsequences appear naturally in connection with norm convergence of multiple ergodic averages \cite{MR2150389}.
The $1$-step nilsequences are exactly the almost periodic sequences.
For examples and a complete description of $2$-step nilsequences see Host, Kra \cite{MR2445824}.
For a characterization of nilsequences of arbitrary step in terms of their local properties see \cite{MR2600993}*{Theorem 1.1}.
Although it is possible to express basic nilsequences as basic nilsequences of the same step associated to ``linear'' sequences of the form $(g^{n})_{n}$ (this is essentially due to Leibman \cite{MR2122919}, see e.g.\ Chu \cite{MR2465660}*{Prop. 2.1} or Green, Tao, Ziegler \cite{2010arXiv1009.3998G}*{Prop. C.2} in the setting of connected Lie groups),
 ``polynomial'' nilsequences, in addition to being formally more general, seem to be better suited for inductive purposes.
This has been observed recently and utilized in connection with additive number theory, see e.g.~Green, Tao, Ziegler \cite{2010arXiv1009.3998G} and Green, Tao \cite{MR2680398}.

From now on we fix a tempered F\o{}lner sequence $(\Phi_N)_{N}$ in $\Z$.
For an ergodic system $(X,\mu,T)$ we denote the Host-Kra factor of order $l$, defined in \cite{MR2150389}, by $\HKZ_{l}(X)$.
We also denote the Sobolev spaces on $G/\Gamma$ by $W^{j,p}(G/\Gamma)$.
All these notions are recalled in Section~\ref{sec:preliminary}.
Our main result, Theorem~\ref{thm:main}, has the following consequence.
\begin{theorem}[Uniform Wiener-Wintner for nilsequences]
\label{thm:uniform-convergence-to-zero}
Assume that $(X,\mu,T)$ is ergodic and let $f\in L^{1}(X)$ be such that $\E(f|\HKZ_{l}(X))=0$.
Let further $G/\Gamma$ be a nilmanifold with a $\Gamma$-rational filtration $\Gb$ on $G$ of length $l$.
Then for a.e. $x\in X$ we have
\begin{equation}
\label{eq:ave-uniform}
\lim_{N\to\infty} \sup_{g\in\poly, F\in W^{k,2^{l}}(G/\Gamma)}
\|F\|_{W^{k,2^{l}}(G/\Gamma)}\inv
\Big| \aveFn f(T^{n}x)F(g(n)\Gamma) \Big| = 0,
\end{equation}
where $k = \sum_{r=1}^{l}(d_{r}-d_{r+1})\binom{l}{r-1}$ with $d_{i}=\dim G_{i}$.

If in addition $(X,T)$ is a uniquely ergodic topological dynamical system and $f\in C(X)\cap \HKZ_l(X)^\bot$ then we have
\begin{equation}
\label{eq:ave-uniform-in-X}
\lim_{N\to\infty} \sup_{g\in\poly, F\in W^{k,2^{l}}(G/\Gamma), x\in X}
\|F\|_{W^{k,2^{l}}(G/\Gamma)}\inv
\Big| \aveFn f(T^{n}x) F(g(n)\Gamma) \Big| = 0.
\end{equation}
\end{theorem}
In view of a counterexample in Section~\ref{sec:counterexample} the Sobolev norm cannot be replaced by the $L^{\infty}$ norm.
On the other hand, we have not investigated whether the above order $k$ is optimal and believe that it is not.

The conclusion \eqref{eq:ave-uniform} differs from the uniform polynomial Wiener-Wintner theorem of Frantzikinakis \cite{MR2246591} in several aspects.
First, our class of weights is considerably more general, comprising all nilsequences rather than polynomial phases (a polynomial phase $f(p(n)\Z)$, $f\in C(\R/\Z)$, $p\in\R[X]$ is also a nilsequence of step $\deg p$ with the filtration $\R = \dots = \R \geq \{0\}$ of length $\deg p$ and cocompact lattice $\Z$).
Also, our result does not require total ergodicity, an assumption that cannot be omitted in the result of Frantzikinakis.
The price for these improvements is that we have to assume the function to be orthogonal to the Host-Kra factor and not only to the Abramov factor of order $l$ (i.e.\ the factor generated by the generalized eigenfunctions of order $\leq l$).

The conclusion \eqref{eq:ave-uniform-in-X} generalizes a result of Assani \cite{MR1995517}*{Theorem 2.10}, which corresponds essentially to the case $l=1$.
Note that without the orthogonality assumption on the function, everywhere convergence can fail even for averages \eqref{ave-ww} for some $\lambda\in \T$.
For more information on this phenomenon we refer to Robinson \cite{MR1271545}, Assani \cite{MR1995517} and Lenz \cite{MR2480747}.

Let $\Gb$ be a $\Gamma$-rational filtration on $G$ and $g\in\poly$ be a polynomial sequence.
By Leibman \cite{MR2122919}*{Theorem B} the sequence $g(n)\Gamma$ is contained and equidistributed in a finite union $\tilde Y$ of sub-nilmanifolds of $G/\Gamma$.
For a Riemann integrable function $F:\tilde Y \to \C$ we call the bounded sequence $(F(g(n)\Gamma))_{n}$ a \emph{basic generalized $l$-step nilsequence} (one obtains the same notion upon replacing the polynomial $g(n)$ by a ``linear'' polynomial $(g^{n})_{n}$).
A \emph{generalized $l$-step nilsequence} is a uniform limit of basic generalized $l$-step nilsequences.

A concrete example of a generalized nilsequence is $(e^{i[n\alpha]n\beta})$ for $\alpha, \beta\in \R$ or, more generally, bounded sequences of the form $(p(n))$ and $(e^{ip(n)})$ for a generalized polynomial $p$, i.e., a function obtained from conventional polynomials using addition, multiplication, and taking the integer part, see Bergelson, Leibman \cite{MR2318563}.

We also obtain an extension of the Wiener-Wintner theorem for nilsequences due to Host and Kra \cite{MR2544760}*{Cor. 2.23 and its proof} to non-ergodic systems.
\begin{theorem}[Wiener-Wintner for generalized nilsequences]
\label{thm:WW-gen-nilseq}
For every $f\in L^1(X,\mu)$ there exists a set $X'\subset X$ of full measure such that for every $x\in X'$ the averages
\begin{equation}\label{ave}
\aveFn a_n f(T^n x)
\end{equation}
converge for every generalized nilsequence $(a_n)$.

If in addition $(X,T)$ is a uniquely ergodic topological dynamical system, $f\in C(X)$ and the projection $\pi:X\to\HKZ_l(X)$ is continuous for some $l$ then the averages \eqref{ave} converge for \emph{every} $x\in X$ and every $l$-step generalized nilsequence $(a_n)$.
\end{theorem}
See Host, Kra and Maass \cite{2012arXiv1203.3778H}*{remarks following Theorem 3.5} for examples of systems for which the additional hypothesis is satisfied.

A consequence of this result concerning norm convergence of weighted polynomial multiple ergodic averages due to Chu \cite{MR2465660}, cf.\  Host, Kra \cite{MR2544760} for the linear case, is discussed in Section~\ref{sec:multiple}.

\begin{ack}
The work on the paper began during the first author's research visit to the University of California, Los Angeles.
She is deeply grateful to her host Terence Tao for many helpful and motivating discussions without which the paper would not have been written.
She thanks UCLA and its analysis group for perfect working conditions and friendly and pleasant atmosphere.
The authors thank Bernard Host and Bryna Kra for their comments, Idris Assani for references and Example~\ref{ex:assani}, Nikos Frantzikinakis for a hint regarding non-ergodic systems and the anonymous referees for corrections and helpful suggestions.
\end{ack}

\section{Notation and tools}\label{sec:preliminary}
We begin with the notions and tools needed.
Throughout the paper we assume an $L^\infty$-function to be defined everywhere.
\begin{definition}[F\o{}lner sequence]
A sequence  $(\Phi_n)$ of finite subsets of a discrete group $G$ is called \emph{F\o{}lner} if for every $g\in G$
\[
\frac{|g\Phi_n\triangle \Phi_n|}{|\Phi_n|}\to 0 \quad \text{as } n\to \infty
\]
holds.
Moreover, a F\o{}lner sequence is called \emph{tempered} (or said to satisfy Shulman's condition) if there exists $C>0$ such that for every $n\in \N$ one has
\[
\Big|\bigcup_{k=1}^n \Phi_k^{-1} \Phi_{n+1}\Big|\leq C|\Phi_{n+1}|.
\]
\end{definition}
Recall that the \emph{maximal function} is defined by
\[
Mf(x) := \sup_{N} \Big| \aveFn f(T^{n}x) \Big|
\text{ for } f\in L^{1}(X).
\]
Lindenstrauss' maximal inequality \cite{MR1865397}*{Theorem 3.2} asserts that for every $f\in L^{1}(X)$ and every $\lambda>0$ we have
\begin{equation}
\label{eq:maximal-inequality}
\mu\{Mf>\lambda\} \lesssim \lambda\inv \|f\|_{1},
\end{equation}
where the implied constant depends only on the constant in the temperedness condition.
\begin{definition}[Generic point]
Let $(\Phi_N)$ be a tempered F\o{}lner sequence in $\Z$, $(X,\mu,T)$ be an ergodic system, and let $f\in
L^\infty(X,\mu)$.
We call $x\in X$ \emph{generic for $f$ with respect to $(\Phi_N)$} if
\[
\aveFn f(T^nx) \to \int_X f \, \dif\mu.
\]
We call $x\in X$ \emph{fully generic for $f$ w.r.t.~$(\Phi_N)$} if it is generic for
every function $g$ in the
(separable)
$T$-invariant subalgebra generated by $f$.
\end{definition}

By a generalization by Lindenstrauss \cite{MR1865397}*{Theorem 1.2} of Birkhoff's ergodic theorem to tempered F\o{}lner sequences, generic and hence fully generic points form a set of full measure.
The temperedness assumption cannot be dropped even for sequences of intervals with growing length in $\Z$, see del Junco, Rosenblatt \cite{MR553340} and Rosenblatt, Wierdl \cite{MR1182661}.
We refer to Butkevich \cite{butkevich} for an overview on pointwise convergence of ergodic averages along F\o{}lner sequences in $\Z$ and general groups, examples and further references.

A measure-preserving system $(X,\mu,T)$ is called \emph{regular} if $X$ is a compact metric space, $\mu$ is a Borel probability measure and $T$ is continuous.
Every measure-preserving system is measurably isomorphic to a regular measure-preserving system upon restriction to a separable $T$-invariant sub-$\sigma$-algebra \cite{MR603625}*{\textsection 5.2}.
The \emph{ergodic decomposition} of the measure on a regular measure-preserving system $(X,\mu,T)$ is a measurable map $x\mapsto\mu_{x}$ from $X$ to the space of $T$-invariant ergodic Borel probability measures on $X$, unique up to equality $\mu$-a.e., such that $\mu$-a.e.\ $x\in X$ is generic for every $f\in C(X)$ w.r.t.\ $\mu_{x}$ and $\mu=\int\mu_{x}\dif\mu(x)$ \cite{MR603625}*{\textsection 5.4}.
Moreover, for every $f\in L^{1}(\mu)$, for $\mu$-a.e.\ $x\in X$  we have that $f\in L^{1}(\mu_{x})$ and $x$ is generic for $f$ w.r.t.\ $\mu_{x}$.

\begin{definition}[Gowers-Host-Kra seminorms]
For a probability measure preserving system $(X,\mu, T)$ and $f\in L^\infty(X,\mu)$,
the \emph{Gowers-Host-Kra seminorms} are defined recursively by
\[
\|f\|_{U^0(X,\mu)}:=\int_X f\dif\mu,\quad
\|f\|_{U^{l+1}(X,\mu)}^{2^{l+1}}:=\limsup_{N\to\infty}\frac{1}{N}\sum_{n=1}^N\|T^n f \bar{f}\|_{U^{l}(X,\mu)}^{2^{l}}.
\]
We will write $U^{l}(X)$ instead of $U^{l}(X,\mu)$ if no confusion is possible.
\end{definition}
These seminorms (that are indeed seminorms for $l\geq 1$) have been introduced by Host and Kra in the ergodic case \cite{MR2150389} and also make sense in the non-ergodic case as pointed out by Chu, Frantzikinakis and Host \cite{MR2795725}.
The limit superior in the above definition is in fact a limit as follows from the characterization of these seminorms via cube spaces \cite{MR2150389}*{\textsection 3.5} and the mean ergodic theorem.
It follows by induction on $l\in\N$ that
\begin{equation}
\label{eq:est-U-by-L}
\|\cdot\|_{U^{l+1}(X)}\leq \|\cdot\|_{L^{2^{l}}(X)},
\end{equation}
see \cite{MR2944094} for subtler analysis.
Moreover, if $\mu = \int \mu_{x} \dif\mu(x)$ is the ergodic decomposition then
\[
\|f\|_{U^{l}(X,\mu)}^{2^{l}}=\int \|f\|_{U^{l}(X,\mu_{x})}^{2^{l}} \dif\mu(x)
\text{ for all } f\in L^{\infty}(\mu).
\]
If $(X,\mu,T)$ is ergodic then for each $l$ there is a factor $\HKZ_l(X)$ of $X$, called the Host-Kra factor of order $l$, that is an inverse limit of $l$-step nilsystems and is such that for all $f\in L^\infty(X)$
\[
\|f\|_{U^{l+1}(X)}=0 \iff \E(f|\HKZ_l(X))=0.
\]
Since the uniformity seminorms are bounded by the supremum norm and invariant under $T$ and complex conjugation they can also be calculated using smoothed averages
\begin{equation}
\label{eq:uniformity-seminorms-smoothed}
\|f\|_{U^{l+1}(X)}^{2^{l+1}}=\lim_{K\to\infty}\frac{1}{K^{2}}\sum_{k=-K}^K (K-|k|)\|T^k f \bar{f}\|_{U^{l}(X)}^{2^{l}}.
\end{equation}
This will allow us to use the following quantitative version of the classical van der Corput estimate (the proof is included for completeness).
Here $o_{K}(1)$ stands for a quantity that goes to zero for each fixed $K$ as $N\to\infty$.
\begin{lemma}[Van der Corput]\label{VdC}
Let $(\Phi_{N})_{N}$ be a F\o{}lner sequence in $\Z$ and $(u_n)_{n\in\Z}$ be a sequence in a Hilbert space with norm bounded by $C$.
Then for every $K>0$ we have
\[
\Big\| \aveFn u_{n} \Big\|^{2} \leq \Big| \frac{2}{K^{2}} \sum_{k=-K}^{K}(K-|k|) \aveFn \langle u_{n},u_{n+k} \rangle \Big| + C^{2} o_{K}(1).
\]
\end{lemma}
\begin{proof}
Let $K>0$ be given.
By the definition of a F\o{}lner sequence we have
\[
\aveFn u_{n}
=
\aveFn \frac{1}{K} \sum_{k=1}^K u_{k+n} + C o_{K}(1).
\]
By Hölder's inequality
\begin{multline*}
\Big\| \aveFn \frac{1}{K} \sum_{k=1}^K u_{k+n} \Big\|^{2}
\leq
\aveFn \Big\|\frac{1}{K} \sum_{k=1}^K u_{k+n}\Big\|^{2}\\
=
\frac{1}{K^{2}} \sum_{k=-K}^{K}(K-|k|) \aveFn \langle u_{n},u_{n+k} \rangle + C^{2} o_{K}(1),
\end{multline*}
and the claim follows using the
estimate $(a+b)^{2}\leq 2a^{2} + 2b^{2}$.
\end{proof}

We now recall the notions of a (nilpotent) (pre-)filtration and a polynomial sequence.
Since in this article we always work in the category of Lie groups we demand all groups in any prefiltration to be Lie.
As mentioned in the introduction, we only consider Lie groups in which the connected component of the identity is simply connected.
\begin{definition}[(Pre-)filtration]
A \emph{prefiltration} $\Gb$ of length $l\in\N=\{0,1,\dots\}$ is a sequence of nested Lie groups
\begin{equation}
\label{eq:prefiltration}
G_{0} \geq G_{1} \geq \dots \geq G_{l+1}=\{1_{G}\}
\quad\text{such that}\quad
[G_{i},G_{j}]\subset G_{i+j}
\quad\text{if }i,j\geq 0,\, i+j\leq l+1.
\end{equation}
The sequence that consists of the trivial group is called the prefiltration of length $-\infty$.
A \emph{filtration} (on a group $G$) is a prefiltration $\Gb$
such that $G_{0}=G_{1}$ (and $G_{0}=G$).
\end{definition}
Although prefiltrations behave well in algebraic constructions, in our analytic arguments we will have to work with filtrations.
Note that in a prefiltration $\Gb$ of length $l$, the subgroup $G_{l}$ need not be central in $G_{0}$.

It is well-known that the lower central series on a nilpotent Lie group $G$ is a filtration on $G$.
If $\Gb$ is a prefiltration of length $l$ and $t\leq l$ then $\Gb[+t]$ denotes the prefiltration of length $l-t$ given by $(\Gb[+t])_{i}=G_{i+t}$.
We will denote the dimension of a Lie group by $d=\dim G$ and the dimension of the $i$-th group in a prefiltration by $d_{i}=\dim G_{i}$.

We define $\Gb$-polynomial sequences by induction on the length of the prefiltration.
\begin{definition}[Polynomial]
\label{def:polynomial}
Let $\Gb$ be a prefiltration of length $l$.
A sequence $g\from\Z\to G_{0}$ is called \emph{$\Gb$-polynomial} if either $l=-\infty$ (so that $g\equiv 1_{G}$) or for every $k\in\Z$ the sequence
\begin{equation}
D_{k}g(n) = g(n)\inv g(n+k)
\end{equation}
is $\Gb[+1]$-polynomial.
We write $\poly$ for the set of $\Gb$-polynomial maps.
\end{definition}
By a result originally due to Leibman \cite{MR1910931} (see \cite{arxiv:1206.0287} for a short proof) the set $\poly$ is in fact a group under pointwise operations and the sequence
\[
\poly \geq \poly[{\Gb[+1]}] \geq \dots \geq \poly[{\Gb[+l+1]}]
\]
is a prefiltration.
We will not need the full strength of this result, but merely that a multiple of a $\Gb$-polynomial sequence and any constant sequence in $G_{0}$ is again $\Gb$-polynomial (this can be easily seen from the definition).

Finally we outline a special case of the cube construction of Green, Tao and Ziegler \cite{2010arXiv1009.3998G}*{Definition B.2} using notation of Green and Tao \cite{MR2877065}*{Proposition 7.2}.
We will only have to perform it on filtrations, but even in this case the result is in general only a prefiltration.
\begin{definition}[Cube construction]
Given a prefiltration $\Gb$ we define the prefiltration $\Gb^{\square}$ by
\[
G^{\square}_{i} := G_{i} \times_{G_{i+1}} G_{i} = \<G_{i}^{\triangle},G_{i+1}\times G_{i+1}\> = \{(g_{0},g_{1})\in G_{i}\times G_{i} : g_{0}\inv g_{1}\in G_{i+1}\},
\]
where $G^{\triangle}=\{(g_{0},g_{1})\in G^{2} : g_{0}=g_{1}\}$ is the diagonal group corresponding to $G$.
By an abuse of notation we refer to the filtration obtained from $\Gb^{\square}$ by replacing $G_{0}^{\square}$ with $G_{1}^{\square}$ as the ``filtration $\Gb^{\square}$''.
\end{definition}
To see that this indeed defines a prefiltration let $x\in G_{i}$, $y\in G_{i+1}$, $u\in G_{j}$, $v\in G_{j+1}$, so that $(x,xy)\in G_{i}^{\square}$ and $(u,uv)\in G_{j}^{\square}$.
Then $[(x,xy),(u,uv)]=([x,u],[xy,uv]) \in G_{i+j}^{\square}$, since $[x,u]\in G_{i+j}$ and
\[
[xy,uv]
=
[x,u] [x,v] [[x,v],[x,u]] [[x,u],v]
[[x,uv],y][y,uv]
\in
[x,u] G_{i+j+1}
\]
(or see \cite{MR2877065}*{Prop.\ 7.2}).
Let now $g\in\poly$.
We show by induction on the length of the prefiltration $\Gb$ that for every $k\in\Z$ the map
\[
g^{\square}_{k}(n) := (g(n+k),g(n))
\]
is $\Gb^{\square}$-polynomial.
Indeed, for $l=-\infty$ there is nothing to show.
If $l\geq 0$ then $g^{\square}_{k}$ takes values in $G^{\square}_{0}$ since $g(n)\inv g(n+k)=D_{k}g(n)\in G_{1}$ by definition of a polynomial.
Moreover $D_{k'}(g^{\square}_{k})=(D_{k'}g)^{\square}_{k}$, so that $D_{k'}(g^{\square}_{k})$ is $\Gb[+1]^{\square}$-polynomial by the induction hypothesis.

As remarked earlier, the prefiltration $\Gb$ and the filtration $\Gb$ are in general distinct concepts.
Also the map $g_{k}^{\square}$ is in general not polynomial with respect to the filtration $\Gb$ since it need not take values in $G_{1}^{\square}$.
However, this is a very mild obstacle and a slight modification of $g^{\square}_{k}$ will work.
A natural candidate is $g_{k}^{\square}(0)\inv g_{k}^{\square}$, but later in the proof this choice would lead to shifts of a function on $G/\Gamma$ by $g(k)$ for every $k$, and there is no useful control on Sobolev norms of such shifts in terms of Sobolev norms of the original function.
Instead we would like to shift only by elements that belong to a fixed compact set and this requires a more sophisticated modification.
\begin{lemma}[Fundamental domain]
\label{lem:fundamental-domain}
Let $\Gamma\leq G$ be a cocompact lattice.
Then there exists a relatively compact set $K\subset G$ and a map $G\to K$, $g\mapsto \{g\}$ such that $g\Gamma = \{g\}\Gamma$ for each $g\in G$.
\end{lemma}
This follows from local homeomorphy of $G$ and $G/\Gamma$, from local compactness of $G$ and from compactness of $G/\Gamma$.
For example, for $G=\R$ and $\Gamma=\Z$ the fundamental domain $K$ can be taken to be the interval $[0,1)$ with the fractional part map $\{ \cdot\}$.
In case of a general connected Lie group the fundamental domain can be taken to be $[0,1)^{\dim(G)}$ in Mal'cev coordinates \cite{MR2877065}*{Lemma A.14}, but we do not need this information.

For each nilmanifold that we consider we fix some map $\{\cdot\}$ as above and define
\begin{equation}
\label{eq:g-tilde}
\tilde g_{k} := (\{g(k)\}\inv g(n+k)g(k)\inv \{g(k)\},\{g(0)\}\inv g(n)g(0)\inv \{g(0)\}).
\end{equation}
This is the conjugate of $g_{k}^{\square} g_{k}^{\square}(0)\inv$ by $(\{g(k)\},\{g(0)\}) \in G_{1}^{2} \subset G_{0}^{\square}$, hence $\Gb^{\square}$-polynomial with values in $G_{1}^{\square}$.

We will use Mal'cev bases adapted to filtrations in the sense of \cite{MR2877065}*{Definition 2.1} with the additional twist that we consider not necessarily connected Lie groups.
This provides additional generality since, by the remark following \cite{MR2445824}*{Theorem 3}, not every nilsequence arises from nilmanifolds associated to connected Lie groups.
\begin{definition}[Mal'cev basis adapted to a filtration]
\label{def:malcev-basis}
Let $G$ be a nilpotent Lie group with a cocompact lattice $\Gamma$ and a filtration $\Gb$ of length $l$ that consists of connected, simply connected Lie groups.
An ordered basis $\{X_{1},\dots,X_{d}\}$ for the Lie algebra of $G$ is called a \emph{Mal'cev basis for $G/\Gamma$ adapted to $\Gb$} if the following conditions are satisfied.
\begin{enumerate}
\item For each $i=1,\dots,l$ the Lie algebra of $G_{i}$ coincides with $\<X_{d-d_{i}+1},\dots,X_{d}\>$.
\item For each $g\in G$ there exist unique numbers $t_{1},\dots,t_{d_{1}}\in\R$, called \emph{Mal'cev coordinates} of $g$, such that $g=\exp(t_{1}X_{1})\dots\exp(t_{d}X_{d})$.
\item The lattice $\Gamma$ consists precisely of the elements with integer Mal'cev coordinates.
\end{enumerate}
\end{definition}
\begin{definition}[Rational filtration]
\label{def:rational-filtration}
We call a filtration $\Gb$ of length $l$ that consists of (not necessarily connected) Lie groups \emph{$\Gamma$-rational} if for every $i=1,\dots,l$ the subgroup $\Gamma_{i}:=\Gamma\cap G_{i}$ is cocompact in $G_{i}$ and there exists a (fixed) Mal'cev basis for $G^{o}/\Gamma^{o}$ adapted to $\Gb^{o}$, where $G^{o}$ denotes the connected component of the identity of a group $G$ and $\Gamma^{o}:=\Gamma\cap G^{o}$.
\end{definition}
The lower central series on a (not necessarily connected) nilpotent Lie group $G$ is $\Gamma$-rational for every cocompact lattice $\Gamma$ \cite{MR0028842}.
In this case Mal'cev coordinates on $G^{o}$ are usually called \emph{coordinates of the second kind}.
Any subfiltration of a rational filtration is clearly rational.
\begin{definition}[Sobolev space]
Let $G/\Gamma$ be a nilmanifold with a $\Gamma$-rational filtration, so in particular we have a Mal'cev basis $\{X_{1},\dots,X_{d}\}$ for the Lie algebra of $G$.
We identify the vectors $X_{i}$ with their extensions to right invariant vector fields on $G/\Gamma$.
The Sobolev space $W^{j,p}(G/\Gamma)$, $j\in\N$, $1\leq p<\infty$, is defined by the norm
\[
\| F \|_{W^{j,p}(G/\Gamma)}^{p}
=
\sum_{a=0}^{j} \sum_{b_{1},\dots,b_{a}=1}^{d} \| X_{b_{1}}\dots X_{b_{a}} F \|_{L^{p}(G/\Gamma)}^{p}.
\]
\end{definition}
Finally, since we will use induction over rational filtrations in the proof of our main result and the inductive hypothesis will involve $\Gb^{\square}$, we have to show that this filtration is rational provided that $\Gb$ is rational.
This follows from the next lemma.
\begin{lemma}[Rationality of the cube filtration]
\label{lem:rational-cube}
Let $\Gb$ be a $\Gamma$-rational filtration.
Then the filtration
\[
G_{0}^{2} = G_{1}^{2} \geq G_{1}^{\square} \geq G_{2}^{2} \geq G_{2}^{\square} \geq
\dots \geq G_{l}^{2} \geq G_{l}^{\square} \geq G_{l+1}^{2} = \{1_{G\times G}\}.
\]
is $\Gamma^{2}$-rational.
In particular, $\Gamma^{\square}=\Gamma^{2} \cap G_{1}^{\square}$ is a cocompact lattice in $G_{1}^{\square}$ and the filtration $\Gb^{\square}$ is $\Gamma^{\square}$-rational.
\end{lemma}
\begin{proof}
Observe first that $(G_{i}^{\square})^{o} = (G^{o})_{i}^{\square}$ for every $i$ since both are closed connected subgroups of $G^{2}$ and their Lie algebras coincide.
The existence (and several additional properties that we do not need) of the required Mal'cev basis follows from a result of Green and Tao \cite{MR2877065}*{Lemma 7.4}.
Clearly, $\Gamma_{i}^{2}$ is cocompact in $G_{i}^{2}$ for every $i=1,\dots,l$.

It remains to show that $\Gamma_{i}^{\square} = \Gamma^{2}\cap G_{i}^{\square}$ is cocompact in $G_{i}^{\square}$ for every $i=1,\dots,l$.
The existence of an adapted Mal'cev basis implies that $\Gamma^{2} \cap (G_{i}^{\square})^{o}$ is cocompact in $(G_{i}^{\square})^{o}$.
Thus it suffices to show that $\Gamma_{i}^{\square} (G_{i}^{\square})^{o}$ has finite index in $G_{i}^{\square}$.
By the assumption $\Gamma_{i}G_{i}^{o}$ has finite index in $G_{i}$ for each $i$, so it contains a finite index normal subgroup $N_{i}\leq G_{i}$ and we can write $G_{i}=A_{i}N_{i}$ with a finite set $A_{i}$.
With this notation we have
\[
G_{i}^{\square}
= G_{i}^{\triangle} G_{i+1}^{2}
= A_{i}^{\triangle} N_{i}^{\triangle} A_{i+1}^{2}N_{i+1}^{2}
\subseteq A_{i}^{\triangle} A_{i+1}^{\triangle} N_{i}^{\triangle} (A_{i+1}\inv A_{i+1} \times \{1_{G}\}) N_{i+1}^{2}.
\]
For every $a\in G_{i+1}$ and $n\in N_{i}$ we have $[a,n\inv]\in G_{i+1} = A_{i+1}N_{i+1}$, so $(n,n)(a,1_{G}) = (a,1_{G}) ([a,n\inv],1_{G}) (n,n) \in (a,1_{G}) (A_{i+1}\times\{1_{G}\}) N_{i+1}^{2} N_{i}^{\triangle}$, so that
\[
G_{i}^{\square}
\subseteq
A_{i}^{\triangle} A_{i+1}^{\triangle} (A_{i+1}\inv A_{i+1} \times \{1_{G}\}) (A_{i+1}\times\{1_{G}\}) N_{i+1}^{2} N_{i}^{\triangle} N_{i+1}^{2},
\]
and since $N_{i+1}^{2} N_{i}^{\triangle} N_{i+1}^{2} \subset \Gamma_{i}^{\square} (G^{o})_{i}^{\square}$ we are done.
\end{proof}

\section{Vertical characters}\label{sec:vertical}
Let $G$ be a nilpotent Lie group with a cocompact lattice $\Gamma$ and a $\Gamma$-rational filtration $\Gb$ of length $l$, so that $\Gamma_{i}=\Gamma\cap G_{i}$ is a cocompact lattice in $G_{i}$ for each $i=1,\dots,l$.
Then $G/\Gamma$ is a smooth principal bundle with the compact abelian Lie structure group $G_{l}/\Gamma_{l}$.
The fibers of this bundle are called ``vertical'' tori (as opposed to the ``horizontal'' torus $G/\Gamma G_{2}$) and everything related to Fourier analysis on them is called ``vertical''.
\begin{definition}[Vertical character]
Let $G/\Gamma$ be a nilmanifold and $\Gb$ a $\Gamma$-rational filtration on $G$.
A measurable function $F$ on $G/\Gamma$ is called a \emph{vertical character} if there exists a character $\chi\in \widehat{G_{l}/\Gamma_{l}}$ such that
for every $g_{l}\in G_{l}$ and a.e.\ $y\in G/\Gamma$
we have $F(g_{l}y)=\chi(g_{l}\Gamma_{l})F(y)$.
\end{definition}
The key ingredient of our proof is the following modification of a construction due to Green and Tao, see e.g.\ \cite{tao-fourier}*{Lemma 1.6.13} and
\cite{MR2877065}*{\textsection 7}, which shows that discrete derivatives of vertical character nilsequences are nilsequences of lower step.

Let $F$ be a smooth vertical character, $g\in\poly$ and $a_{n}=F(g(n)\Gamma)$ be the corresponding basic nilsequence.
A calculation shows that for every $k\in\Z$ we have
\[
a_{n+k} \ol{a_{n}}
= (F \otimes \ol{F})(g^{\square}_{k}(n)\Gamma^{2})
= \left(\{g(k)\}F \otimes \{g(0)\}\ol{F} \right)(\tilde g_{k}(n)\Gamma^{2})
=: \tilde F_{k} (\tilde g_{k}(n) \Gamma^{\square}),
\]
where $\tilde F_{k}$ is the restriction of $\{g(k)\}F\otimes \{g(0)\}\ol{F}$ from $G_{1}^{2}$ to $G_{1}^{\square}$.
Recall that the filtration $\Gb^{\square}$ is $\Gamma^{\square}$-rational by Lemma~\ref{lem:rational-cube}.
Since $F$ is a vertical character, $\{g(k)\}F\otimes \{g(0)\}\ol{F}$ is $G^{\square}_{l}$-invariant (note that $G^{\square}_{l}=G^{\triangle}_{l}$), so that $\tilde F_{k}$ is well-defined on $\tilde Y=\tilde G_{1}/\tilde\Gamma_{1}$, where $\tilde G = G_{1}^{\square}/G_{l}^{\square}$ is a nilpotent group with the cocompact lattice $\tilde\Gamma=\Gamma^{\square}G^{\square}_{l}/G^{\square}_{l}$ and the $\tilde\Gamma$-rational filtration $\tilde G_{i}=G^{\square}_{i}/G^{\square}_{l}$, $i=1,\dots,l$, $\tilde G_{0}=\tilde G_{1}$.
Abusing the notation we may consider $\tilde g_{k}(n)$ also as an element of $\poly[\tilde \Gb]$, so
\[
a_{n+k}\ol{a_{n}} = \tilde F_{k} (\tilde g_{k}(n) \tilde\Gamma)
\]
is a basic nilsequence of step $l-1$.

We will write $A\lesssim_{D} B$ if $A$ and $B$ satisfy the inequality $A\leq CB$ with some constant $C$ that depends on some auxiliary constant(s) $D$ and some geometric data.
\begin{lemma}[Control on Sobolev norms in the cube construction]
\label{lem:sobolev-norm-of-tensor-product}
With the above notation we have
\begin{equation}
\label{eq:estimate-discrete-derivative}
\|\tilde F_{k}\|_{W^{j,p}(\tilde{G}/\tilde{\Gamma})}\lesssim_{j,p} \|F\|^{2}_{W^{j,2p}(G/\Gamma)} \text{ for any } j\in\N, p\in [1,\infty),
\end{equation}
where the implied constant does not depend on $k$ and $F$.
\end{lemma}
\begin{proof}
For the Mal'cev basis on $\tilde G/\tilde\Gamma$ that is induced by the Mal'cev basis on $G_{1}^{\square}/\Gamma^{\square}$ we have
\[
\|\tilde F_{k}\|_{W^{j,p}(\tilde G/\tilde\Gamma)} = \|\{g(k)\}F\otimes\{g(0)\}\ol{F}\|_{W^{j,p}(G_{1}^{\square}/\Gamma^{\square})},
\]
so it suffices to estimate the latter quantity.

For this end observe that the Haar measure on $G_{1}^{\square}/\Gamma^{\square}$ is a self-joining of the Haar measure on $G/\Gamma$ under the canonical projections to the coordinates.
Therefore and by the Cauchy-Schwarz inequality we have
\begin{align*}
\|F_{0}\otimes F_{1}\|_{L^p(G_{1}^{\square}/\Gamma^{\square})}^{2p}
&= \Big(\int_{G_{1}^{\square}/\Gamma^{\square}}|F_{0}(y_{0}) F_{1}(y_{1})|^{p} \dif\mu_{G_{1}^{\square}/\Gamma^{\square}}(y_{0},y_{1})\Big)^2\\
&\leq \int_{G_{1}^{\square}/\Gamma^{\square}} |F_{0}(y_{0})|^{2p} \dif\mu_{G_{1}^{\square}/\Gamma^{\square}}(y_{0},y_{1}) \int_{G_{1}^{\square}/\Gamma^{\square}} |F_{1}(y_{1})|^{2p} \dif\mu_{G_{1}^{\square}/\Gamma^{\square}}(y_{0},y_{1})\\
&= \int_{G/\Gamma} |F_{0}|^{2p} \dif\mu_{G/\Gamma}  \int_{G/\Gamma} |F_{1}|^{2p} \dif\mu_{G/\Gamma}
=\|F_{0}\|_{L^{2p}(G/\Gamma)}^{2p} \|F_{1}\|_{L^{2p}(G/\Gamma)}^{2p}
\end{align*}
for any smooth functions $F_{0},F_{1}$ on $G/\Gamma$.
Now recall that $\{g(k)\}\in K$ for some fixed compact set $K\subset G_{1}$, so that by smoothness of the group operation $\|\{g(k)\}F\|_{L^{2p}(G/\Gamma)} \lesssim \|F\|_{L^{2p}(G/\Gamma)}$, and analogously for $\{g(0)\}\ol{F}$.
Similar calculations for the derivatives lead to the bound
\[
\|\{g(k)\}F\otimes\{g(0)\}\ol{F}\|_{W^{j,p}(G_{1}^{\square}/\Gamma^{\square})} \lesssim_{j,p} \|F\|_{W^{j,2p}(G/\Gamma)}^{2}.
\qedhere
\]
\end{proof}

\begin{definition}[Vertical Fourier series]
Let $G/\Gamma$ be a nilmanifold and $\Gb$ be a $\Gamma$-rational filtration on $G$.
For every $F\in L^{2}(G/\Gamma)$ and $\chi\in\widehat{G_{l}/\Gamma_{l}}$ let
\begin{equation}
\label{eq:vert-char-repr}
F_{\chi}(y):=\int_{G_{l}/\Gamma_{l}}F(g_{l}y) \ol{\chi}(g_{l}) \dif g_{l}.
\end{equation}
\end{definition}
With this definition $F_{\chi}$ is
defined almost everywhere
and is a vertical character as witnessed by the character $\chi$.
The usual Fourier inversion formula implies that $F=\sum_{\chi\in\widehat{G_{l}/\Gamma_{l}}}F_{\chi}$ in $L^{2}(G/\Gamma)$.
We further need the following variant of Bessel's inequality.
\begin{lemma}[Bessel-type inequality for vertical Fourier series]\label{lem:bessel}
Let $p\in[2,\infty)$ and $F\in L^{p}(G/\Gamma)$.
Then
\[
\sum_\chi \|F_\chi\|_{L^p(G/\Gamma)}^p\leq \|F\|_{L^p(G/\Gamma)}^p.
\]
\end{lemma}
Note that the analogue for $p=\infty$ follows immediately from \eqref{eq:vert-char-repr}.
\begin{proof}
Since vertical characters have constant absolute value on $G_l/\Gamma_{l}$-fibers, we have by \eqref{eq:vert-char-repr} and the Cauchy-Schwarz inequality
\begin{align*}
\|F_\chi\|_{L^p(G/\Gamma)}^p
&=
\int_{G/\Gamma} \int_{G_l/\Gamma_{l}} |F_\chi(h h_l)|^2 \dif h_l \cdot |F_\chi(h)|^{p-2}\, \dif h\\
&\leq \int_{G/\Gamma} \int_{G_l/\Gamma_{l}}   |F_\chi(hh_l)|^2 \dif h_l  \Big(\int_{G_l/\Gamma_{l}} |F(hh_l)|^2 \dif h_l\Big)^{p/2 -1} \dif h
\end{align*}
for every $\chi$.
By the Plancherel identity and H\"older's inequality this implies
\begin{align*}
\sum_{\chi} \|F_{\chi}\|_{L^p(G/\Gamma)}^p
&\leq \int_{G/\Gamma} \Big( \int_{G_l/\Gamma_{l}} |F(hh_l)|^2 \dif h_l\Big)^{p/2} \dif h \\
&\leq \int_{G/\Gamma} \int_{G_l/\Gamma_{l}} |F(hh_l)|^p \dif h_l \dif h
= \|F\|_{L^p(G/\Gamma)}^p,
\end{align*}
finishing the proof.
\end{proof}
It is worth mentioning that there is also a Plancherel-type identity
\[
\sum_\chi \|F_\chi\|_{U^{l}(G/\Gamma)}^{2^{l}}=\|F\|_{U^{l}(G/\Gamma)}^{2^{l}}
\]
for Gowers-Host-Kra norms and vertical Fourier series, see Eisner, Tao \cite{MR2944094}*{Lemma 10.2} for the case $l=3$.

\begin{lemma}[Control on Sobolev norms in a vertical Fourier series]
\label{lem:estimate-vertical-fourier-series}
Let $j\in\N$ and $p\in[2,\infty)$.
For every smooth function $F$ on $G/\Gamma$ we have
\[
\sum_\chi \|F_\chi\|_{W^{j,p}(G/\Gamma)}\lesssim_{j,p} \|F\|_{W^{j+d_{l},p}(G/\Gamma)}.
\]
\end{lemma}
\begin{proof}
The compact abelian Lie group $G_{l}/\Gamma_{l}$ is isomorphic to a product of a torus and a finite group.
In order to keep notation simple we will consider the case $G_{l}/\Gamma_{l} \cong \T^{d_{l}}$, the conclusion for disconnected $G_{l}/\Gamma_{l}$ follows easily from the connected case.
We rescale the last $d_{l}$ elements of the Mal'cev basis in such a way that they correspond to the unit tangential vectors at the origin of the torus $\T^{d_{l}}$.
The characters on $G_{l}/\Gamma_{l}$ are then given by $\chi_{\m}(z_{1},\dots,z_{m})=z_{1}^{m_{1}}\cdot\dots\cdot z_{d_{l}}^{m_{d_{l}}}$ with $\m=(m_{1},\dots,m_{d_{l}})\in\Z^{d_{l}}$.
Observe that by (\ref{eq:vert-char-repr}) and the commutativity of $G_l$ we have $(\partial_{i} F)_\m=\partial_{i}(F_\m)=m_iF_\m$ for every $i$ and $\m$, where $\partial_{i}$ denotes the derivative along the $i$-th coordinate in $\T^{d_{l}}$.
Therefore, by H\"older's inequality and Lemma~\ref{lem:bessel}
\begin{align*}
\Big(\sum_{m_1, \ldots, m_{d_l}\neq 0}\|F_{\chi_{\m}} \|_{L^{p}}\Big)^{p}
&= \Big(\sum_{m_1, \ldots, m_{d_l}\neq 0} \frac{1}{|m_1\cdots m_{d_l}|}\|m_1\cdots m_{d_l} F_{\chi_{\m}} \|_{L^{p}}\Big)^{p}\\
&\leq \Big(\sum_{m_1, \ldots, m_{d_l}\neq 0} \Big|\frac{1}{m_1\cdots m_{d_l}}\Big|^{p/(p-1)} \Big)^{p-1} \sum_\m \|m_1\cdots m_{d_l}  F_{\chi_{\m}} \|_{L^{p}}^{p} \\
&\lesssim \sum_{\m} \|\partial_{1}\ldots \partial_{d_l} F_{\chi_{\m}}\|_{L^{p}}^{p}
\leq \|\partial_{1}\ldots \partial_{d_l} F\|_{L^{p}}^{p}
\leq \|F\|_{W^{d_l,p}}^{p}.
\end{align*}
By the centrality of $G_{l}$ the operations of taking derivatives along elements of the Mal'cev basis and taking the $\chi$-th vertical character \eqref{eq:vert-char-repr} commute, so we have
\[
\sum_{m_1, \ldots, m_{d_l}\neq 0}\|F_{\chi_{\m}}\|_{W^{j,p}}
\lesssim \|F\|_{W^{j+d_l,p}}
\]
for every $j\in \N$.
The same argument works if some of the indices $(m_1,\ldots,m_{d_l})$ vanish, in which case a smaller number of derivatives is added to $j$, and thus altogether
\[
\sum_\m\|F_{\chi_{\m}}\|_{W^{j,p}}
\lesssim \|F\|_{W^{j+d_l,p}}.
\qedhere
\]
\end{proof}

We will need an estimate on the $L^{\infty}$ norm of a vertical character in terms of a Sobolev norm with minimal smoothness requirements.
For this end we would like to use a Sobolev embedding theorem on $G/\Gamma G_{l}$ since this manifold has lower dimension than $G/\Gamma$.
Morally, a vertical character is a function on the base space $G/\Gamma G_{l}$ that is extended to the principal $G_{l}/\Gamma_{l}$-bundle $G/\Gamma$ in a multiplicative fashion.
However, in general this bundle lacks a global cross-section, so we are forced to work locally.
\begin{lemma}[Sobolev embedding]\label{lem:sobolev-embedding}
Let $G/\Gamma$ be a nilmanifold and $\Gb$ a $\Gamma$-rational filtration of length $l$ on $G$.
Then for every $1\leq p\leq\infty$ and every vertical character $F\in W^{d-d_{l},p}(G/\Gamma)$ we have
\[
\|F\|_{\infty} \lesssim_{p} \|F\|_{W^{d-d_{l},p}},
\]
where the implied constant does not depend on $F$.
\end{lemma}
\begin{proof}
The case $p=\infty$ is clear, so we may assume $p<\infty$.

Since $\Gamma$ is discrete there exists a neighborhood $U \subset G$ of the identity such that the quotient map $U\to G/\Gamma$ is a diffeomorphism onto its image.
Let $M\subset G$ be a $(d-d_{l})$-dimensional submanifold that intersects $G_{l}$ in $e_{G}$ transversely.
By joint continuity of multiplication in $G$ we may find neighborhoods of identity $V\subset G_{l}$ and $W\subset M$ such that $VW\subset U$.
By transversality the differential of the map $\psi:V\times W \to G$, $(v,w)\mapsto vw$ is invertible at $(e_{G},e_{G})$, so by the inverse function theorem and shrinking $V,W$ if necessary we may assume that $\psi$ is a diffeomorphism onto its image.
We may also assume that $V,W$ are connected, simply connected and have smooth boundaries.
Recalling that the quotient map $U\to G/\Gamma$ is a diffeomorphism, we obtain a chart $\Psi:V\times W\to G/\Gamma$ for a neighborhood of $e_{G}\Gamma$ that has the additional property that $\Psi(g_{l}v,w)=g_{l}\Psi(v,w)$ whenever $v,g_{l}v\in V$.
Shrinking $V$ and $W$ further if necessary we may assume that the differential of $\Psi$ and its inverse are uniformly bounded.
By homogeneity we obtain similar charts for some neighborhoods of all points of $G/\Gamma$.
By compactness $G/\Gamma$ can be covered by finitely many such charts, so it suffices to estimate $\|F\|_{L^{\infty}(\im\Psi)}$ in terms of $\|F\|_{W^{d-d_{l},p}(\im\Psi)}$.

By definition of Sobolev norms we have
\[
\int_{v\in V} \|F\circ\Psi\|_{W^{d-d_{l},p}(\{v\}\times W)}^{p} \dif v \lesssim \|F\circ\Psi\|_{W^{d-d_{l},p}(V\times W)}^{p} \lesssim \|F\|_{W^{d-d_{l},p}(\im\Psi)}^{p}.
\]
Since $F$ is a vertical character and by multiplicativity of $\Psi$ in the first argument, the integrand on the left-hand side is constant, so that
\[
\|F\circ\Psi\|_{W^{d-d_{l},p}(\{v\}\times W)} \lesssim \|F\|_{W^{d-d_{l},p}(\im\Psi)} \text{ for all } v\in V,
\]
the bound being independent of $v$.
Now, $W$ is a $d-d_{l}$ dimensional manifold, so the usual Sobolev embedding theorem \cite{MR2424078}*{Theorem 4.12 Part I Case A} applies and we obtain
\[
\|F\circ\Psi\|_{L^{\infty}(\{v\}\times W)} \lesssim \|F\circ\Psi\|_{W^{d-d_{l},p}(\{v\}\times W)} \lesssim \|F\|_{W^{d-d_{l},p}(\im\Psi)}^{p}.
\]
By the above discussion this implies the desired estimate.
\end{proof}

\section{The main estimate}\label{sec:estimate}
In this section we deal with our main problem of estimation of averages in \eqref{eq:ave-uniform}.
The general strategy is to decompose $F$ into a vertical Fourier series, to use the quantitative van der Corput estimate and to control various norms that appear during this procedure using the results of the previous section.
In several places
in our argument we will need convergence of Birkhoff averages of a function to its integral.
In order to ensure this convergence we restrict attention to fully generic points.

The following uniform estimate is our main result.
\begin{theorem}[Uniformity seminorms control averages uniformly]\label{thm:main}
Assume that $(X,\mu,T)$ is ergodic.
Then for every $f\in L^\infty(X)$ and every point $x$ that is fully generic for $f$ with respect to $(\Phi_N)$ the following holds.
For every $l\in\N$ and $\epsilon>0$ there exists $N_{0}$ such that for every nilmanifold $G/\Gamma$ with a $\Gamma$-rational filtration $\Gb$ on $G$ of length $l$, every smooth function $F$ on $G/\Gamma$ and every $g\in\poly$ we have
\begin{equation}\label{eq:control-by-uniformity-norms}
\forall N\geq N_{0}
\quad
\Big| \aveFn f(T^nx) F(g(n)\Gamma) \Big|
\lesssim \|F\|_{W^{k,2^l}(G/\Gamma)} (\|f\|_{U^{l+1}(X)} + \epsilon),
\end{equation}
where $k=\sum_{r=1}^{l}(d_{r}-d_{r+1})\binom{l}{r-1}$ and the implied constant depends only on the nilmanifold $G/\Gamma$, filtration $\Gb$ and the Mal'cev basis that is implicit in the definition of $\Gamma$-rationality.

If in addition $(X,T)$ is uniquely ergodic and $f\in C(X)$ then the conclusion holds for every $x\in X$ and $N_{0}$ can be chosen independently of $x$.
\end{theorem}

Example~\ref{ex:assani} below shows that there is in general no constant $C$ such that the estimate
\begin{equation}\label{eq:assani}
\limsup_{N\to\infty}\Big|\aveN f(T^nx)F(S^ny)\Big|\leq C \|F\|_\infty \|f\|_{U^2(X)}
\end{equation}
holds for every $1$-step basic nilsequence $F(S^ny)$, even without uniformity.
Thus one cannot expect to replace the Sobolev norm by $\|F\|_\infty$ in Theorem~\ref{thm:main}.

\begin{remark}\label{remark-hk-ww}
Quantifying the proof of Host, Kra \cite{MR2150389}*{Proposition 5.6} using standard Fourier analysis on $\T^{d(2^l-1)}$ one obtains for the non-uniform averages the upper bound
\[
\limsup_{N} \Big| \aveFn f(T^{n}x) F(g(n)\Gamma) \Big| \lesssim \|F\|_{W^{d(2^l-1),2}(G/\Gamma)} \|f\|_{U^{l+1}(X)}
\]
for ``linear'' sequences $g(n)=h^{n}h'$, where the implied constant depends on geometric data like the choice of a decomposition of identity on the pointed cube space $(G/\Gamma)^{[k]}_*=(G/\Gamma)^{2^l-1}$.
Note also that Host and Kra worked with intervals with growing length instead of tempered F\o{}lner sequences in $\Z$.
\end{remark}

\begin{proof}[Proof of Theorem~\ref{thm:main}]
We argue by induction on $l$.
In the case $l=0$ the group $G$ is trivial, so $\|F\|_{\infty} = \|F\|_{W^{0,1}(G/\Gamma)}$ and the claim follows by the definition of generic points.
We now assume that the claim holds for $l-1$ and show that it holds for $l$.
Write $a_n:=F(g(n)\Gamma)$.

Assume first that $F$ is a vertical character and recall the notation from Section~\ref{sec:vertical}.
Let $\delta>0$ be chosen later.
For the dimensions $(\tilde d_{i})$ of the groups in the filtration $\tilde\Gb$
we have the relations
$\tilde d_{i}-\tilde d_{i+1}=(d_{i}-d_{i+1})+(d_{i+1}-d_{i+2})$, $i=1,\dots,l-1$.
By the induction hypothesis applied to $\tilde{G}/\tilde{\Gamma}$ with the induced $\tilde\Gamma$-rational filtration and Lemma~\ref{lem:sobolev-norm-of-tensor-product} we have
\begin{align*}
\Big| \aveFn (T^kf \bar{f})(T^nx) a_{n+k}{\ol{a_n}}\Big|
&\lesssim \|\tilde{F}_{k}\|_{W^{\tilde k,2^{l-1}}} (\|T^kf \bar{f}\|_{U^{l}(X)} + \delta)\\
&\lesssim \|F\|_{W^{\tilde k,2^{l}}}^2 (\|T^kf \bar{f}\|_{U^{l}(X)} + \delta)
\end{align*}
with $\tilde k = \sum_{r=1}^{l-1}(\tilde d_{r}-\tilde d_{r+1}) \binom{l-1}{r-1} = \sum_{r=1}^{l}(d_{r}-d_{r+1}) \binom{l}{r-1} - d_{l}$
for any integer $k$ provided that $N$ is large enough depending on $l$, $k$, $\delta$ and $x$.
Let $K$ be chosen later.
The van der Corput Lemma~\ref{VdC} implies
\begin{align*}
\Big| \aveFn f(T^nx) a_n\Big|^2
&\leq
\frac2{K^{2}} \sum_{k=-K}^{K}(K-|k|)
\Big| \aveFn (T^kf \bar{f})(T^nx) a_{n+k}{\ol{a_n}}\Big|
+\|F\|_{\infty}^{2}\|f\|_{\infty}^{2}o_{K}(1)\\
&\lesssim
\frac1{K^{2}} \sum_{k=-K}^{K}(K-|k|) \|F\|_{W^{\tilde k,2^{l}}}^2 (\|T^kf \bar{f}\|_{U^{l}(X)} + \delta)
 +\|F\|_{\infty}^{2} o_{K}(1)
\end{align*}
provided that $N$ is large enough depending on $l$, $K$, $\delta$ and $x$.
By Lemma~\ref{lem:sobolev-embedding} this is dominated by
\[
\|F\|_{W^{\tilde k,2^{l}}}^2 \left(\frac1{K^{2}} \sum_{k=-K}^{K}(K-|k|) \|T^kf \bar{f}\|_{U^{l}(X)} + \delta + o_{K}(1)\right).
\]
By the Cauchy-Schwarz inequality this is dominated by
\[
\|F\|_{W^{\tilde k,2^{l}}}^2 \Big( \Big(\frac1{K^{2}} \sum_{k=-K}^{K}(K-|k|) \|T^kf \bar{f}\|_{U^{l}(X)}^{2^{l}}\Big)^{1/2^{l}} + \delta +o_{K}(1) \Big)=:I.
\]
By \eqref{eq:uniformity-seminorms-smoothed} for sufficiently large $K = K(f,\delta)$ the above average over $k$ approximates $\|f\|_{U^{l+1}(X)}^{2}$ to within $\delta$, so we have
\begin{align*}
I &\lesssim
\|F\|_{W^{\tilde k,2^{l}}}^2 (\|f\|_{U^{l+1}(X)}^{2} + 2\delta +o_{K}(1)).
\end{align*}
Taking $\delta=\delta(\epsilon)$ sufficiently small and $N\geq N_{0}(l,f,\epsilon,x)$ sufficiently large we obtain
\[
\Big| \aveFn f(T^nx) a_n\Big|
\lesssim
\|F\|_{W^{\tilde k,2^{l}}} (\|f\|_{U^{l+1}(X)} + \epsilon).
\]
Note that $N_{0}$ does not depend on $F$.

Let now $(a_n)=(F(g(n)\Gamma))$ be an arbitrary $l$-step basic nilsequence on $G/\Gamma$.
Let $F=\sum_\chi F_\chi$ be the vertical Fourier series.
By the above investigation of the vertical character case, since the vertical Fourier series of $F$ converges absolutely and by Lemma~\ref{lem:estimate-vertical-fourier-series} we get
\begin{align*}
\Big| \aveFn f(T^n x) F(g(n)\Gamma) \Big|
&\lesssim
\sum_\chi \|F_\chi\|_{W^{\tilde k,2^l}} (\|f\|_{U^{l+1}(X)} + \epsilon)\\
&\lesssim
\|F\|_{W^{\tilde k+d_{l},2^{l}}} (\|f\|_{U^{l+1}(X)} + \epsilon)
\end{align*}
for $N\geq N_{0}$ as required.

Under the additional assumptions that $(X,T)$ is uniquely ergodic and $f\in C(X)$ we obtain the additional conclusion that the estimate is uniform in $x\in X$ for $l=0$ from uniform convergence of ergodic averages $\aveFn T^{n}f$, see e.g.~\cite{MR648108}*{Theorem 6.19}.
For general $l$ it suffices to observe that in the above proof the dependence of $N_{0}$ on $x$ comes in only through the inductive hypothesis.
Also, there is no need for temperedness of $(\Phi_{N})$ in this case.
\end{proof}

\begin{proof}[Proof of Theorem~\ref{thm:uniform-convergence-to-zero}]
Let $f\in L^{1}(X)$ with $\E(f|\HKZ_{l}(X))=0$ be given.
By truncation we can approximate it by a sequence of bounded functions $(f_{j})\subset L^{\infty}(X)$ such that $f_{j}\to f$ in $L^{1}$.
Replacing each $f_{j}$ by $f_{j}-\E(f_{j}|\HKZ_{l}(X))$ we may assume that $\E(f_{j}|\HKZ_{l}(X))=0$ for every $j$.

By Theorem~\ref{thm:main} we have
\[
\lim_{N\to\infty} \sup_{g\in\poly, F\in W^{k,2^{l}}(G/\Gamma)}
\|F\|_{W^{k,2^{l}}(G/\Gamma)}\inv
\Big| \aveFn f_{j}(T^{n}x)F(g(n)\Gamma) \Big| = 0
\]
for $x$ in a set of full measure and every $j$.
By the Sobolev embedding theorem \cite{MR2424078}*{Theorem 4.12 Part I Case A} we have $\|F\|_{\infty} \lesssim \|F\|_{W^{k,2^{l}}(G/\Gamma)}$ for $F\in W^{k,2^{l}}(G/\Gamma)$.
This shows that
\begin{multline*}
\sup_{g\in\poly, F\in W^{k,2^{l}}(G/\Gamma)}
\|F\|_{W^{k,2^{l}}(G/\Gamma)}\inv
\Big| \aveFn f(T^{n}x)F(g(n)\Gamma) \Big|\\
\lesssim
\aveFn |f-f_{j}|(T^{n}x) +
\sup_{g\in\poly, F\in W^{k,2^{l}}(G/\Gamma)} \|F\|_{W^{k,2^{l}}(G/\Gamma)}\inv \Big| \aveFn f_{j}(T^{n}x)F(g(n)\Gamma) \Big|.
\end{multline*}
Fixing a $j$, restricting to the set of points that are generic for $|f-f_{j}|$ with respect to~$\{\Phi_N\}$ and letting $N\to\infty$ we can estimate the limit by $\|f-f_{j}\|_{1}$ pointwise on a set of full measure.
Hence the limit vanishes a.e.

Under the additional assumptions that $(X,T)$ is uniquely ergodic and $f$ is continuous the uniform convergence \eqref{eq:ave-uniform-in-X} follows directly from Theorem~\ref{thm:main}.
\end{proof}

\section{A counterexample}\label{sec:counterexample}
The following example shows that there is no constant $C$ such that the estimate (\ref{eq:assani})
holds for every $1$-step basic nilsequence $F(S^ny)$. Thus one cannot replace the Sobolev norm by $\|F\|_\infty$ in Theorem \ref{thm:main} even without uniformity in $F$ and $g$.
\begin{example}[I.~Assani]\label{ex:assani}
We begin as in Assani, Presser \cite{MR2901351}*{Remarks} and consider an irrational rotation system $(\T,\mu,T)$ on the unit circle, $f\in C(\T)$, $x\in \T$ and define $S:=T$, $y:=x$ and $F:=\bar{f}$.
We have
\[
\limsup_{N\to\infty}\Big|\aveN f(T^nx)\bar{f}(T^nx) \Big|
= \sum_{k=-\infty}^\infty |\hat{f}(k)|^2=\|f\|_2^2.
\]
By $\|f\|_{U^2(\T)}^4=\sum_{k=-\infty}^\infty |\hat{f}(k)|^4$, the inequality \eqref{eq:assani} takes the form
\begin{equation}\label{eq:assani-2}
\|f\|_2^2
\leq C\|f\|_\infty \Big(\sum_{k=-\infty}^\infty |\hat{f}(k)|^4 \Big)^{1/4}.
\end{equation}
Let now $\{a_n\}_{n=1}^\infty\subset \R$ and consider random polynomials
\[
P_N(t,\omega) :=\sum_{n=1}^N r_n(\omega) a_n \cos(nt),
\]
where $r_n$ are the Rademacher functions taking the values $1$ and $-1$ with equal probability. By Kahane \cite{MR833073}*{pp. 67--69}, there is an absolute constant $D$ such that for every $N$
\[
\mathbb{P}\Big\{\omega:\, \|P_N(\cdot,\omega)\|_\infty \geq D\Big(\sum_{n=1}^N a_n^2 \log N \Big)^{1/2}\Big\}\leq \frac{1}{N^2}.
\]
Therefore for every $N\in\N$ there is $\omega$ (or a choice of signs $+$ or $-$) so that
\[
\|P_N(\cdot,\omega)\|_\infty\leq D\Big(\sum_{n=1}^N a_n^2 \log N \Big)^{1/2}.
\]

Assume now that inequality \eqref{eq:assani-2} holds for some constant $C$ and every $f\in C(\T)$. Then by the above for $f=P_N(\cdot,\omega)$ we have
\[
\sum_{n=1}^N a_n^2
\leq CD(\log N)^{1/2}
\Big(\sum_{n=1}^N a_n^2\Big)^{1/2} \Big(\sum_{n=1}^N a_n^4\Big)^{1/4}
\]
and hence
\[
\sum_{n=1}^N a_n^2 \leq (CD)^2 \|(a_n)\|_{l^4} \log N.
\]
Taking $a_n=\sqrt{\log n/n}$
implies $\sum_{n=1}^N \log n/n\leq \tilde{C} \log N$ for some $\tilde{C}$ and all $N$, a contradiction.
\end{example}
We also refer to Assani \cite{MR2753294} and Assani, Presser \cite{MR2901351} for related issues.

\section{Wiener-Wintner theorem for generalized nilsequences}\label{sec:WW-gen-nilseq}
In view of Theorem \ref{thm:main} the Wiener-Wintner theorem for generalized nilsequences (Theorem~\ref{thm:WW-gen-nilseq}) follows by a limiting argument from a structure theorem for non-ergodic measure preserving systems due to Chu, Frantzikinakis and Host.
\begin{proof}[Proof of Theorem~\ref{thm:WW-gen-nilseq}]
Restricting to the separable $T$-invariant $\sigma$-algebra generated by $f$ 
we may assume that $(X,\mu,T)$ is regular.
Let $\mu = \int \mu_{x} \dif\mu(x)$ be the ergodic decomposition.

Consider first a function $0\leq f\leq 1$ and let $\tilde f:=\E(f|\HKZ_{l}(X))$.
By \cite{MR2795725}*{Proposition 3.1} we obtain a sequence of functions $(f_{j}) \subset L^{\infty}(X)$ such that the following holds.
\begin{enumerate}
\item We have $\|f_{j}\|_{L^{\infty}(X,\mu)} \leq 1$ and $\|\tilde f-f_{j}\|_{L^{1}(X,\mu)} \to 0$ as $j\to\infty$.
\item For every $j$ and $\mu$-a.e.\ $x\in X$ the sequence $(f_{j}(T^{n}x))_{n}$ is an $l$-step nilsequence.
\end{enumerate}
Using the first condition we can pass to a subsequence such that $\|\tilde f-f_{j}\|_{L^{2^{l-1}}(X,\mu_{x})} \to 0$ for a.e.\ $x\in X$.
Thus we obtain a full measure subset $X'\subset X$ such that the following holds for every $x\in X'$:
\begin{enumerate}
\item for every $j$ the sequence $(f_{j}(T^{n}x))_{n}$ is an $l$-step nilsequence,
\item for every $j$ the point $x$ is fully generic for $f-f_{j}$ with respect to an ergodic measure $\mu_{x}$ and
\item $\|f-f_{j}\|_{U^{l}(X,\mu_{x})} \to 0$ as $j\to\infty$ (this follows from the basic inequality \eqref{eq:est-U-by-L}).
\end{enumerate}
Let $x\in X'$ and $(a_{n})$ be a basic $l$-step nilsequence of the form $a_{n}=F(g(n)\Gamma)$ with smooth $F$.
Since the product of two nilsequences is again a nilsequence, by Leibman
\cite{MR2122919}*{Theorem A} the limit
\[
\lim_{N\to\infty} \aveFn f_j(T^n x) F(g(n)\Gamma)
\]
exists for every $j\in\N$.
By Theorem~\ref{thm:main} we have
\[
\limsup_{N\to\infty} \Big| \aveFn (f-f_{j})(T^{n}x)F(g(n)\Gamma) \Big| \lesssim \|f-f_{j}\|_{U^{l}(X,\mu_{x})}
\]
for every $j$, where the constant does not depend on $j$, and this implies the existence of the limit \eqref{ave}.

Let now $x\in X'$ and $(a_{n})$ be a basic generalized nilsequence of the form $a_n=F(g(n)\Gamma)$ with a real valued Riemann integrable function $F$.
Let $\veps>0$.
Since $F$ is Riemann integrable on $\tilde Y=\ol{\{g(n)\Gamma : n\in\Z\}}$ (which is a finite union of sub-nilmanifolds with the weighted Haar measure $\nu$) and by the Tietze extension theorem, there exist continuous functions $F_\veps$ and $H_\veps$ on $G/\Gamma$ with $F_\veps\leq F\leq H_\veps$ such that $\int (H_\veps - F_\veps) \dif\nu < \veps$.
By mollification we may assume that $H_\veps$ and $F_\veps$ are smooth.
By the above the limits $\lim_{N} \aveFn f(T^{n}x) H_{\veps}(g(n)\Gamma)$ and $\lim_{N} \aveFn f(T^{n}x) F_{\veps}(g(n)\Gamma)$ exist.
By continuity of $F_\veps$ and $H_\veps$ we have for every $x\in X'$
\begin{multline*}
\left(\limsup_{N\to\infty}- \liminf_{N\to\infty}\right)\aveFn f(T^nx) F(g(n)\Gamma)\\
\leq \lim_{N\to\infty} \aveFn f(T^nx) (H_\veps - F_\veps)(g(n)\Gamma)\\
\leq \lim_{N\to\infty} \aveFn (H_\veps - F_\veps)(g(n)\Gamma)
=  \int_{\tilde Y} (H_\veps - F_\veps) \dif\nu <  \veps,
\end{multline*}
and since $\veps>0$ was arbitrary this proves the existence of the limit \eqref{ave}.

A limiting argument allows one to replace the basic generalized nilsequence by a generalized nilsequence.
By linearity we obtain the conclusion for $f\in L^{\infty}(X)$.
The general case $f\in L^{1}(X)$ follows from the maximal inequality \eqref{eq:maximal-inequality}.

Under the additional assumptions of unique ergodicity of $(X,T)$ and continuity of the projection $\pi:X\to\HKZ_{l}(X)$ we find that the functions $f_{j}$ can be chosen to be continuous on $X$ by \cite{MR2600993}*{Theorem A} and every point is fully generic for $f-f_{j}$, allowing us to replace the set of full measure $X'$ in the above argument by $X$.
\end{proof}

\section{Weighted multiple averages}\label{sec:multiple}
The Wiener-Wintner theorem (Theorem~\ref{thm:WW-gen-nilseq} for linear nilsequences) was used by Host and Kra \cite{MR2544760}*{Theorem 2.25} to show that the values of a bounded measurable function along almost every orbit of an ergodic transformation are good weights for $L^{2}$ convergence of linear multiple ergodic averages.
A polynomial extension of this result was proved by Chu \cite{MR2465660}*{Theorem 1.1}.
Since our Theorem~\ref{thm:WW-gen-nilseq} is stated for ``polynomial'' nilsequences we can slightly shorten the proof of her result that we formulate for $L^{1}$ functions and tempered F\o{}lner sequences.

\begin{corollary}[Convergence of weighted multiple ergodic averages]
\label{cor:weighted-multiple-convergence}
Let $(\Phi_N)$ be as above and let $\phi \in L^1(X)$.
Then there is a set $X'\subset X$ of full measure such that for every $x\in X'$ the sequence $\phi(T^n x)$ is a \emph{good weight for polynomial multiple ergodic averages along $(\Phi_N)$}, i.e., for every measure-preserving system $(Y,\nu, S)$, integer polynomials $p_{1},\dots,p_{k}$ and functions $f_1,\ldots,f_k\in L^\infty(Y,\nu)$ the averages
\begin{equation}
\label{eq:weighted-multiple}
\aveFn \phi(T^nx) S^{p_{1}(n)} f_1\cdots S^{p_{k}(n)} f_k
\end{equation}
converge in $L^2(Y,\nu)$ as $N\to \infty$.
\end{corollary}
In order to reduce to an appropriate nilfactor we need the following variant of \cite{MR2465660}*{Theorem 2.2}.
Recall that two polynomials are called \emph{essentially distinct} if their difference is not constant.
\begin{lemma}
\label{lem:weighted-multiple-zero}
Let $(\Phi_{N})_{N}$ be an arbitrary F\o{}lner sequence in $\Z$.
For every $r,d\in\N$ there exists $k\in\N$
such that for every ergodic system $(X,\mu,T)$, any functions $f_{1},\dots,f_{r} \in L^{\infty}(X)$ with $\|f_{1}\|_{U^{k}(X)}=0$,
any non-constant pairwise essentially distinct integer polynomials $p_{1},\dots,p_{r}$ of degree at most $d$ and any bounded sequence of complex numbers $(a_{n})_{n}$ we have
\[
\limsup_{N\to\infty} \Big\| \aveFn a_{n} T^{p_{1}(n)}f_{1} \cdots T^{p_{r}(n)}f_{r} \Big\|_{L^{2}(X)} = 0.
\]
\end{lemma}
\begin{proof}
We may assume that $(a_{n})$ is bounded by $1$.
By a variant of the van der Corput lemma \cite{MR2151605}*{Lemma 4} there exists a F\o{}lner sequence $(\Theta_{M})$ in $\Z^{3}$ such that the square of the left-hand side is bounded by
\begin{multline*}
\limsup_{M} \frac{1}{|\Theta_{M}|} \Big| \sum_{(n,v,w)\in\Theta_{M}} a_{n+v}\ol{a_{n+w}} \int_{X} \prod_{i=1}^{r} T^{p_{i}(n+v)}f_{i} T^{p_{i}(n+w)}\ol{f_{i}} \Big|\\
\leq
\limsup_{M} \frac{1}{|\Theta_{M}|} \sum_{(n,v,w)\in\Theta_{M}} \Big| \int_{X} \prod_{i=1}^{r} T^{p_{i}(n+v)}f_{i} T^{p_{i}(n+w)}\ol{f_{i}} \Big|.
\end{multline*}
By the Cauchy-Schwarz inequality the square of this expression is bounded by
\begin{multline*}
\limsup_{M} \frac{1}{|\Theta_{M}|} \sum_{(n,v,w)\in\Theta_{M}} \Big| \int_{X} \prod_{i=1}^{r} T^{p_{i}(n+v)}f_{i} T^{p_{i}(n+w)}\ol{f_{i}} \Big|^{2}\\
=
\limsup_{M} \frac{1}{|\Theta_{M}|} \sum_{(n,v,w)\in\Theta_{M}} \int_{X\times X} \prod_{i=1}^{r} (T\times T)^{p_{i}(n+v)}(f_{i}\otimes \ol{f_{i}}) (T\times T)^{p_{i}(n+w)}(\ol{f_{i}}\otimes f_{i}).
\end{multline*}
Let $\mu\times\mu = \int_{s\in Z} (\mu\times\mu)_{s} \dif s$ be the ergodic
decomposition of $\mu\times\mu$.
By Fatou's lemma the above expression is bounded by
\begin{multline*}
\int_{s\in Z}\limsup_{M} \frac{1}{|\Theta_{M}|} \sum_{(n,v,w)\in\Theta_{M}} \int_{X\times X} \prod_{i=1}^{r} (T\times T)^{p_{i}(n+v)}(f_{i}\otimes \ol{f_{i}}) (T\times T)^{p_{i}(n+w)}(\ol{f_{i}}\otimes f_{i}) \dif(\mu\times\mu)_{s}\, \dif s\\
\leq
\int_{s\in Z}\limsup_{M} \Big\| \frac{1}{|\Theta_{M}|} \sum_{(n,v,w)\in\Theta_{M}} \prod_{i=1}^{r} (T\times T)^{p_{i}(n+v)}(f_{i}\otimes \ol{f_{i}}) (T\times T)^{p_{i}(n+w)}(\ol{f_{i}}\otimes f_{i}) \Big\|_{L^{1}(X\times X,(\mu\times\mu)_{s})}\, \dif s.
\end{multline*}
Convergence to zero of the integrand follows from Leibman \cite{MR2151605}*{Theorem 3} provided that $\|f_{1} \otimes \ol{f_{1}}\|_{U^{k-1}(X\times X,(\mu\times\mu)_{s})}=0$ for some sufficiently large $k$.
It follows from Host, Kra \cite{MR2150389}*{Lemma 3.1} and the original definition of the uniformity seminorms in \cite{MR2150389}*{\textsection 3.5} that
\[
\|f_{1}\|_{U^{k}(X)}^{2^{k}} = \int_{s\in Z} \|f_{1} \otimes \ol{f_{1}}\|_{U^{k-1}(X\times X,(\mu\times\mu)_{s})}^{2^{k-1}}\, \dif s.
\]
Thus the hypothesis ensures convergence to zero of the integrand in the previous display for a.e.\ $s$ provided that $k$ is large enough.
\end{proof}
\begin{proof}[Proof of Corollary~\ref{cor:weighted-multiple-convergence}]
By ergodic decomposition it suffices to consider ergodic systems $(Y,\nu,S)$.

Assume first that $\phi\in L^{\infty}(X)$.
By Lemma~\ref{lem:weighted-multiple-zero} we may assume that each $f_{i}$ is measurable with respect to some Host-Kra factor $\HKZ_{l}(Y)$.

By density we may further assume that each $f_{i}$ is a continuous function on a nilsystem factor of $Y$.
In this case the sequence $S^{p_{i}(n)}f_{i}(y)$ is a basic nilsequence of step at most $l \deg p_{i}$ for each $y\in Y$, and the product $\prod_{i} S^{p_{i}(n)}f_{i}(y)$ is also a basic nilsequence of step at most $l \max_{i}\deg p_{i}$.
Therefore the averages \eqref{eq:weighted-multiple} converge pointwise on $Y$ for a.e.\ $x\in X$ by Theorem~\ref{thm:WW-gen-nilseq}, and by the Dominated Convergence Theorem they converge in $L^{2}(Y)$.

We can finally pass to $\phi\in L^{1}(X)$ using the maximal inequality \eqref{eq:maximal-inequality}.
\end{proof}

\bibliography{pzorin-ergodic-MR,pzorin-ergodic-preprints}
\end{document}